\documentclass[a4paper,11pt]{article}

\usepackage[T1]{fontenc}
\usepackage[english]{babel}
\usepackage{amsmath,amssymb}
\usepackage{amsmath,amssymb,amsthm,amsfonts}
\usepackage[margin=3.7cm]{geometry}
\usepackage{eucal}
\usepackage{graphicx}

\usepackage{tikz}

\usepackage{fancyhdr}
 \newtheorem{theorem}{Theorem}
\newtheorem{lemma}[theorem]{Lemma}
\newtheorem{proposition}[theorem]{Proposition}

 \theoremstyle{remark}

\theoremstyle{definition}

\numberwithin{theorem}{section}
\numberwithin{equation}{section}

 \newcommand{\tri}{\mathcal{T}}
 \newcommand{\R}{\mathbb{R}}

 \newcommand{\osc}{\mathrm{osc}}
 \newcommand{\cor}{\mathcal{C}}
 \newcommand{\energy}{\mathcal{E}}
 
 \newcommand{\corv}{V_H}
 \newcommand{\vS}{v_H}
 \newcommand{\uH}{u_H}
 \newcommand{\vH}{v_H}
 \newcommand{\vh}{v_h}
 
 \newcommand{\new}[1]{#1}

\begin{document}

\author{Daniel Peterseim\thanks{Institut f\"ur Numerische Simulation, 
            Universit\"at Bonn, 
            Wegelerstra{\ss}e 6, D-53115 Bonn, Germany;
             peterseim@ins.uni-bonn.de,
             schedensack@ins.uni-bonn.de}
         \and Mira Schedensack\footnotemark[3]} 
\title{Relaxing the CFL condition for the wave equation on adaptive 
       meshes\thanks{The final publication is available at link.springer.com.}
        \thanks{D.~Peterseim gratefully acknowledges support by the Hausdorff 
                Center for Mathematics Bonn and by Deutsche Forschungsgemeinschaft in the 
                Priority Program 1748 ``Reliable simulation techniques in solid mechanics: 
                Development of non-standard discretization methods, mechanical 
                and mathematical analysis'' under the project
                ``Adaptive isogeometric modeling of propagating strong 
                discontinuities in heterogeneous materials''.
                The authors would like to thank Andreas Longva for pointing out 
                that mass lumping indeed works.
                Parts of this paper were written while the authors enjoyed the 
                kind hospitality of the Hausdorff Institute for Mathematics (Bonn).}
          }

\date{}
 
 \maketitle
 
\begin{abstract}
\noindent
The Courant-Friedrichs-Lewy (CFL) condition guarantees the 
stability of the popular explicit leapfrog method for the wave equation.
However, it limits the choice of the time step size to be bounded by the minimal 
mesh size in the spatial finite element mesh. This essentially prohibits 
any sort of adaptive mesh refinement that would be required to reveal optimal 
convergence rates on domains with re-entrant corners.
This paper shows how a simple subspace projection step inspired by 
numerical homogenisation can remove the critical time step restriction so that 
the CFL condition and approximation properties are balanced in an 
optimal way, even in the presence of spatial singularities.
% \keywords{CFL condition \and hyperbolic equation \and 
% finite element method \and adaptive mesh refinement}
% \subclass{65M12 \and 65M60 \and 35L05}
\end{abstract}
 


\section{Introduction}

We consider the discretisation of the wave equation
\begin{equation}\label{e:strongform}
\begin{aligned}
 \ddot{u} - \Delta u &= f &&\text{in }(0,T)\times\Omega,\\
 u(0) &= u_0 \qquad&&\text{in }\Omega,\\
 \dot{u}(0)&=v_0 &&\text{in }\Omega,\\
 u\vert_{\partial\Omega} &= 0 &&\text{in }(0,T)
\end{aligned}
\end{equation}
on a polygonal, bounded Lipschitz
domain $\Omega\subseteq\R^d$, $d\in\{2,3\}$, with (possibly) re-entrant corners. 
This typically 
reduces the regularity of the solution and leads to $u(t)\not\in H^2(\Omega)$. 
To reveal optimal convergence rates, non-uniform mesh refinement in space 
proves advantageous for the wave equation~\cite{MuellerSchwab2015}.

The spatial discretisation with linear finite elements (or any other suitable 
Ritz-Galerkin method) based on some regular 
triangulation $\tri_h$ of $\Omega$ turns problem~\eqref{e:strongform} into a system of ordinary 
differential equations. 
Explicit central difference schemes are very popular 
for the discretisation of this time-dependent system. 
In the context of $H^1$-conforming finite elements, explicit means that the 
scheme avoids the expensive application of the inverse finite element stiffness 
matrix in every time step as it would be required by implicit backward differences. 
Only the mass matrix, which is well-conditioned after suitable diagonal scaling,
needs to be inverted. In particular, no sort of advanced multi-level 
preconditioning is necessary. This is one reason for the wide use of central 
differences. Another reason is that 
the symmetry of central differences leads to conservation of the 
inherent energy of the problem. Among the most simple and successful schemes of 
this type is the 
leapfrog, also known as second order explicit Newmark's scheme and 
St\"ormer-Verlet method. 

As usual for explicit time discretisation schemes, the numerical stability is 
conditional and guaranteed only under the sharp 
Courant-Friedrichs-Lewy (CFL) condition~\cite{CourantFriedrichsLewy1927}. 
In the present context of linear finite elements, it bounds the time step 
size by the minimal mesh size of the spatial mesh
\begin{align*}
  \Delta t\lesssim h_\mathrm{min}\,.
\end{align*}
While on (quasi-)uniform meshes the admissible choice 
$\Delta t \approx h_\mathrm{min}\approx h_\mathrm{max}$ 
is considered as a natural balance of space and time discretisation, 
the CFL condition is not at all compatible with non quasi-uniform meshes in the sense 
that the efficiency of adaptive mesh refinement in space causes 
tiny time steps that destroy the overall complexity. Essentially, the CFL condition 
forbids any type of spatial adaptivity.

The aim of this paper is to show that this phenomenon is a consequence of the 
high flexibility of adaptive finite elements. The restriction of the time step 
by the minimal spatial mesh size can easily be removed by projecting the 
adaptive finite element space to some subspace $\corv$ with similar (optimal)
approximation properties for weak solutions of the wave equation under the  
moderate regularity assumption $\ddot{u}(t)\in L^2(\Omega)$ for almost all $t$. 
The underlying technique is well-established in the context of numerical 
homogenisation~\cite{MalqvistPeterseim2014_localization}, also for 
a semi-discrete wave equation~\cite{AbdulleHenning2015}.
The reduced 
space $\corv$ allows for an improved inverse inequality that decouples the time step 
from the minimal mesh size
and turns the leapfrog into a feasible numerical 
scheme also on adaptive spatial meshes.
\new{The basis functions of the reduced space $\corv$ have to be computed 
and do not have in general a local support. These additional costs 
can be reduced by a localisation approach, see Section~\ref{s:praxis}.
Moreover, in the numerical experiment in Section~\ref{s:numerics},
the combination of the proposed method with mass lumping still shows 
the optimal convergence rate. This turns the method in a fully explicit
scheme.}

Another approach for avoiding (global) fine time step sizes consists in 
a combination of fine time step sizes in regions with small spatial 
elements and of larger time step sizes in regions with coarser spatial 
elements. This approach was introduced in \cite{DiazGrote2009} and is
motivated by small geometric features.
It seems to 
work very well in the case of locally isolated refinement and essentially two 
separate spatial discretisation scales. 
However, adaptive triangulations arising from spatial singularities are 
typically graded towards the singularity 
and encounter an ever increasing number of spatial discretisation scales. Although the  
generalisation of local time stepping to this case with an increasing number 
of different time step sizes 
is possible \cite{DiazGrote2015}, its realisation is certainly challenging 
and its behaviour with regard to stability and computational 
complexity is still open in such scenarios. Our aim is to provide an 
alternative approach for the stabilisation of explicit time stepping that is 
based on reduction of spatial degrees of freedom
rather than enriching the temporal discretisation.

Other approaches to overcome a strong CFL condition is the 
locally implicit method analysed in~\cite{HochbruckSturm2015}, which 
combines an explicit method with an implicit method in the region, 
where the mesh-size is small, and the singular complement 
method~\cite{CiarletHe2003}, which adds singular functions to the 
standard ansatz space.

\smallskip
The remaining parts of this paper are organised as follows. 
Section~\ref{s:idearelax} defines a generalised finite element space and proves 
optimal approximation properties and the improved inverse inequality 
in Lemma~\ref{l:approxinverse}. 
Section~\ref{s:discretisation} introduces the discretisation of the wave 
equation and states an error estimate. Section~\ref{s:praxis} discusses some 
practical aspects and generalisations of the method, while Section~\ref{s:numerics} concludes 
the paper with a numerical experiment.

\smallskip 
Standard notation on Lebesgue and Sobolev spaces is employed throughout the 
paper and $\|\bullet\|:=\|\bullet\|_{L^2(\Omega)}$ abbreviates the $L^2$~norm 
over $\Omega$, while $(\bullet,\bullet)_{L^2(\Omega)}$ denotes the 
$L^2$~scalar product. The notation $\bullet$ abbreviates the identity mapping.
The space $L^2(0,T;X)$ denotes the space of Bochner square integrable 
functions from $[0,T]$ with values in $X$. The dual pairing between 
$f\in H^{-1}(\Omega)$ and $v\in H^1_0(\Omega)$ is denoted by 
$\langle f,v\rangle_{H^{-1}(\Omega)\times H^1_0(\Omega)}$.
The symbol $C$ denotes a generic constant which is independent of the 
mesh size.

\section{Spatial reduction}\label{s:idearelax}

This section recalls the CFL condition \new{for the leapfrog discretisation
from Section~\ref{s:discretisation} below} in the context of adaptive (spatial)
finite elements and presents our novel reduction technique.

\subsection{CFL condition, inverse inequality and approximation}
\label{ss:inverseIneq}

Given a shape regular triangulation $\tri_h$, let $S^1_0(\tri_h)$ denote the standard 
$P_1$-FEM space of $\tri_h$-piece\-wise affine and globally continuous functions, which 
vanish on $\partial\Omega$.
The precise CFL condition for the leapfrog discretisation with underlying 
finite element space $S^1_0(\tri_h)$ reads
\begin{align}\label{e:CFLintro}
  \Delta t \leq \frac{\sqrt{2}}{C_\mathrm{inv}(S^1_0(\tri_h))},
\end{align}
where $C_\mathrm{inv}(S^1_0(\tri_h))$ is the best constant in the inverse inequality
in $S^1_0(\tri_h)$, i.e.,
\begin{align*}
  \|\nabla \vh\| \leq C_\mathrm{inv}(S^1_0(\tri_h)) \| \vh\|
   \qquad\text{for all }\vh\in S^1_0(\tri_h).
\end{align*}
In other words, the inverse inequality constant 
$C_\mathrm{inv}(S^1_0(\tri_h))^2$ is the maximal Rayleigh quotient 
$ \|\nabla \vh\|^2_{L^2(\Omega)}/\| \vh\|_{L^2(\Omega)}^2$ amongst all shape 
functions $v_h\in S^1_0(\tri_h)$ and, hence,  $C_\mathrm{inv}(S^1_0(\tri_h))^2$ 
is the largest eigenvalue of the discrete Laplacian in the sense of finite 
elements; see also Subsection~\ref{ss:eigenvalues} below. 
It is well known and easy to see that it
scales like the reciprocal of the minimal mesh size $h_\mathrm{min}$ of the 
underlying finite element mesh $\tri_h$, 
$$C_\mathrm{inv}(S^1_0(\tri_h))\leq C h_\mathrm{min}^{-1}.$$ 
Thus, the CFL condition~\eqref{e:CFLintro} says that the time step must not exceed some fixed 
multiple of the minimal spatial mesh size.
This 
suggests the use of quasi-uniform meshes in space. However, in the presence 
of singularities, quasi-uniform meshes $\tri_H$ with mesh size $H$ 
lead to the suboptimal best approximation error 
\begin{align}\label{e:approxVH}
%  \sup_{\substack{u\in V,\\ \Delta u\in L^2(\Omega)}}
  \inf_{\vS\in S^1_0(\tri_H)} \|(u-\vS)\|_{H^1(\Omega)}
   \leq C H^\delta\|\Delta u\|
 \quad \text{for all }u\in V\text{ with }\Delta u\in L^2(\Omega),
\end{align}
where $V:=H^1_0(\Omega)$ and $\delta<1$ depends on the domain $\Omega$.

The following Subsection~\ref{ss:constructioncorv} constructs a 
generalised finite element space $\corv$ with 
$\mathrm{dim}(\corv)=\mathrm{dim}(S^1_0(\tri_H))$ such that the 
(quasi-uniform) mesh size $H$ of $\tri_H$ satisfies simultaneously the (optimal) 
approximation property~\eqref{e:approxVH} with $\delta=1$ and the inverse 
inequality 
\begin{align}\label{e:inverseineq}
 \|\nabla v_H\|
 \leq C_\mathrm{inv}(\corv) \|v_H\|
 \qquad\text{for all }v_H\in\corv
\end{align}
with $C_\mathrm{inv}(\corv)\leq C H^{-1}$.
Provided $\Delta t \leq \widetilde{C} H$, this allows for the stability of explicit 
time stepping schemes without losing optimal approximation properties in space.

\subsection{Construction of reduced space}\label{ss:constructioncorv}

We consider a quasi-uniform shape regular triangulation $\tri_H$ with (maximal) mesh size $H$ 
and some (possibly) non-quasi-uniform shape regular triangulation and refinement $\tri_h$ of $\tri_H$ with 
corresponding finite element spaces $S^1_0(\tri_H)$ and $S^1_0(\tri_h)$ and with 
approximation property
\begin{align}\label{e:approxVh}
%  \sup_{\substack{u\in V,\\ \Delta u\in L^2(\Omega)}}
   \inf_{\vh\in S^1_0(\tri_h)} \|u-\vh\|_{H^1(\Omega)}
   \leq C H\|\Delta u\|
 \quad \text{for all }u\in V\text{ with }\Delta u\in L^2(\Omega).
\end{align}
The construction of the generalised finite element space is based on a projective 
quasi-interpolation operator $I_H:V\to S^1_0(\tri_H)$ with  
approximation and stability properties 
\begin{align}\label{e:IHapproxstab}
 \|H^{-1}(v-I_H v)\|
   + \|\nabla I_H v\|
  \leq C_{I_H}^{(1)} \|\nabla v\|
  \quad\text{for all }v\in V
\end{align}
and the $L^2$~stability
\begin{align}\label{e:IHL2stab}
 \|I_H v\|
   \leq C_{I_H}^{(0)} \|v\|
   \quad\text{for all }v\in V.
\end{align}
While~\eqref{e:IHapproxstab} is a standard property of quasi-interpolations,
the $L^2$-stability~\eqref{e:IHL2stab} is not, e.g., the Scott-Zhang 
quasi-interpolation~\cite{ScottZhang1990} is not $L^2$ stable.
For an admissible projective quasi-interpolation, which satisfies 
both~\eqref{e:IHapproxstab} and~\eqref{e:IHL2stab}, one may think of 
the $L^2$ projection onto $S^1_0(\tri_H)$, which is $H^1$ stable on 
quasi-uniform meshes~\cite{BankDupont1981}.
Another example for this Cl\'ement-type quasi-interpolation is given in 
Subsection~\ref{ss:exampleIH} below. 
Denote the kernel of $I_H$ as $W_h:=\operatorname{ker}(I_H\vert_{S^1_0(\tri_h)})\subseteq S^1_0(\tri_h)$.
Given $\vS\in S^1_0(\tri_H)$, define the projection $\cor \vS\in W_h$ of $\vS$ 
onto $W_h$ by 
\begin{align}\label{e:corprob}
  (\nabla w_h,\nabla \cor \vS)_{L^2(\Omega)}
  = (\nabla w_h,\nabla \vS)_{L^2(\Omega)}
 \qquad\text{for all }w_h\in W_h.
\end{align}
Section~\ref{s:praxis} below discusses the efficient computation of this 
projection, e.g., by localisation.
Define the space $\corv$ by $\corv:=(1-\cor)S^1_0(\tri_H)$, which  
implies 
\begin{align}\label{e:decompVh}
  S^1_0(\tri_h) = W_h \oplus \corv 
\end{align}
and the sum is orthogonal with respect to 
$(\nabla\bullet,\nabla\bullet)_{L^2(\Omega)}$.
The following lemma proves that the inverse 
inequality~\eqref{e:inverseineq} holds with constant $C_\mathrm{inv}(\corv)\leq C H^{-1}$ 
independent of the minimal mesh size in $\tri_h$. Moreover, a direct 
consequence of this lemma is that the approximation property~\eqref{e:approxVh} 
is preserved in the coarse space $\corv$.

\begin{lemma}\label{l:approxinverse}
There exists a constant $C_\mathrm{appr}$ such that for all 
$u\in V$ with $\Delta u\in L^2(\Omega)$ and all $u_h\in S^1_0(\tri_h)$ with 
$(\nabla u_h,\nabla \vh)_{L^2(\Omega)}=(-\Delta u,\vh)_{L^2(\Omega)}$
for all $\vh\in S^1_0(\tri_h)$, it holds
\begin{align}\label{e:approxcorv}
  \inf_{\vH\in \corv} \|u_h-\vH\|_{H^1(\Omega)}
   \leq C_\mathrm{appr} H\|\Delta u\|.
\end{align} 
Furthermore, the constant $C_\mathrm{inv}(\corv)$ from~\eqref{e:inverseineq} 
satisfies $C_\mathrm{inv}(\corv)\leq C H^{-1}$.
\end{lemma}

\begin{proof}
Let $\uH\in \corv$ be the Galerkin projection of $u_h$ onto $\corv$, i.e., 
\begin{align*}
  (\nabla \uH,\nabla \vH)_{L^2(\Omega)} 
    = (\nabla u_h,\nabla \vH)_{L^2(\Omega)} 
    \qquad\text{for all }\vH\in \corv.
\end{align*}
Set $e_h:=u_h-\uH\in S^1_0(\tri_h)$.
The Galerkin orthogonality
\begin{align*}
 (\nabla e_h,\nabla \vH)_{L^2(\Omega)} = 0 
   \qquad\text{for all }\vH\in \corv
\end{align*}
and the orthogonality of the subspace decomposition \eqref{e:decompVh} imply that $e_h\in W_h$. 
The approximation properties~\eqref{e:IHapproxstab} of $I_H$ therefore lead to
\begin{align*}
 \|e_h\|
   = \|e_h - I_H e_h\|
   \leq C_{I_H}^{(1)} H \|\nabla e_h\|
\end{align*}
and, hence, 
\begin{align*}
 \|\nabla e_h\|^2 
   = (\nabla e_h,\nabla e_h)_{L^2(\Omega)} 
   &= (\nabla u_h,\nabla e_h)_{L^2(\Omega)} 
   = (-\Delta u, e_h)_{L^2(\Omega)}\\
  &\leq C_{I_H}^{(1)} H \|\Delta u\|\; \|\nabla e_h\|.
\end{align*}
This proves the approximation property \eqref{e:approxcorv}.

For the proof of the inverse inequality let $\vH\in\corv$.
Since $(1-\cor)$ is a projection onto $V_H$ and $(1-I_H)$ is a projection into 
$W_h$, it is easily seen that 
$$(1-\cor)I_H \vH = (1-\cor)\vH - (1-\cor)(1-I_H) \vH = \vH.$$
The orthogonality of $(1-\cor)$ with respect to 
$(\nabla\bullet,\nabla\bullet)_{L^2(\Omega)}$, hence, leads to 
\begin{align*}
  \|\nabla \vH\|
 = \|\nabla (1-\cor)I_H \vH\|
 \leq \|\nabla I_H \vH\|.
\end{align*}
The classical inverse inequality in $S^1_0(\tri_H)$ \cite{BrennerScott08} and 
the $L^2$~stability of $I_H$ from 
\eqref{e:IHL2stab} lead to 
\begin{align*}
 \|\nabla I_H \vH\|
 \leq C H^{-1} \|I_H \vH\|
 \leq C C_{I_H}^{(0)} H^{-1} \|\vH\|.
\end{align*}
The combination of the previous two inequalities concludes the proof.
\end{proof}

In the case of an adaptive refinement $\tri_h$ of $\tri_H$,
Lemma~\ref{l:approxinverse} indicates that the reduced space $V_H$ approximates 
any function $u\in V$ with $\Delta u\in L^2(\Omega)$ indeed with the same rate 
as the full space $S^1_0(\tri_h)$. 
We shall now try to relate the corresponding 
approximation errors more explicitly using techniques from a~posteriori error 
analysis. Let 
$$\osc(\tri,\Delta u):=\|h_\tri(\Delta u - \Pi^\tri_0 \Delta u)\|$$ denote the
oscillations of $\Delta u$ with respect to a triangulation $\tri$, where 
$\Pi^\tri_0$ is the $L^2$~projection  onto $\tri$-piecewise constant functions
and $h_\tri$ is the piecewise constant mesh size function.
\new{Let $\lvert T\rvert$ denote the area of a triangle for $d=2$ or the 
volume of a tetrahedron for $d=3$. 
For a function $f$ that is constant on $T\in\tri_H$, we have 
\begin{align*}
 \|H f\|_{L^2(T)} = (f\vert_T)\, H\, \lvert T\rvert^{1/2}
  \leq C(\tri_H,\tri_h)\, (f\vert_T) \sqrt{\sum_{T\supseteq K\in\tri_h} h_K^2\, \lvert K\rvert}
  = C(\tri_H,\tri_h) \,\|h_{\tri_h} f\|_{L^2(T)}
\end{align*}
with 
\begin{align*}
  C(\tri_H,\tri_h):=\max_{T\in\tri_H} 
   \left( H\lvert T\rvert^{1/2} 
    \Big/\max_{T\supseteq K \in\tri_h} (h_K\lvert K\rvert^{1/2})\right). 
\end{align*}}%
Since $\tri_h$ is a refinement of $\tri_H$, it holds that 
\new{$h_{\tri_h}\vert_K\leq H$ for all $K\in \tri_h$} and we have 
\begin{align*}
 \| H\Delta u\| 
  \leq \|H \Pi^{\tri_H}_0\Delta u\| + \osc(\tri_H,\Delta u)
   \leq C\new{(\tri_H,\tri_h)} \|h_{\tri_h} \Delta u\| 
     + (C\new{(\tri_H,\tri_h)}+1) \osc(\tri_H,\Delta u).
\end{align*}
If, e.g., only triangles at a corner singularity are refined, the constant 
$C\new{(\tri_H,\tri_h)}$ is 
uniformly bounded independent of the mesh sizes. The efficiency 
$$\|h_{\tri_h}\Delta u\|\leq C_\mathrm{eff} \|\nabla(u-u_h)\| 
+ \osc(\tri_h,\Delta u)$$ from a~posteriori error analysis \cite{Verfuerth96}
then proves together with a triangle inequality, 
Lemma~\ref{l:approxinverse} and C\'ea's lemma
\begin{align*}
 \|\nabla (u-\uH)\|
   \leq C \big( \|\nabla(u-u_h)\| + \osc(\tri_H,\Delta u)\big)
\end{align*}
for the Galerkin projection $\uH$ of $u$ in $\corv$\new{, where $C$ depends 
on $C(\tri_H,\tri_h)$}.
This means that the Galerkin approximation of $u$ in $\corv$ is comparable 
with that in $S^1_0(\tri_h)$ up to oscillations.

\subsection{Illustration by finite element eigenvalues}
\label{ss:eigenvalues}

\begin{figure}
 \begin{center}
   \includegraphics[width=0.49\textwidth]{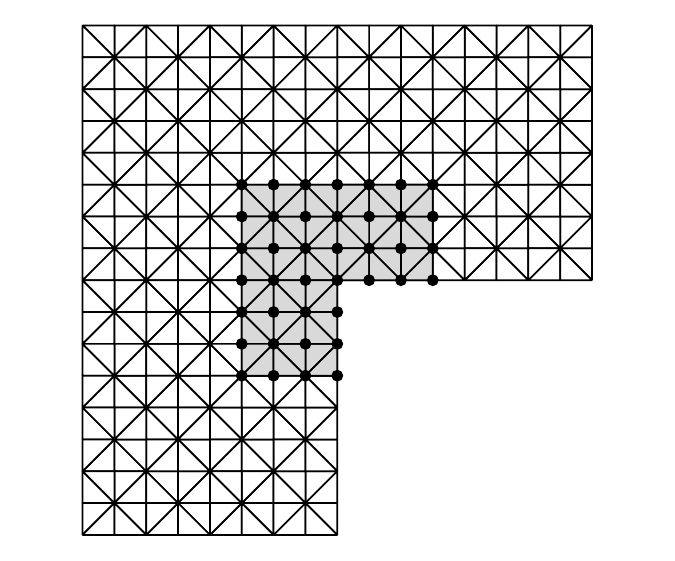}
 \hfill
   \includegraphics[width=0.49\textwidth]{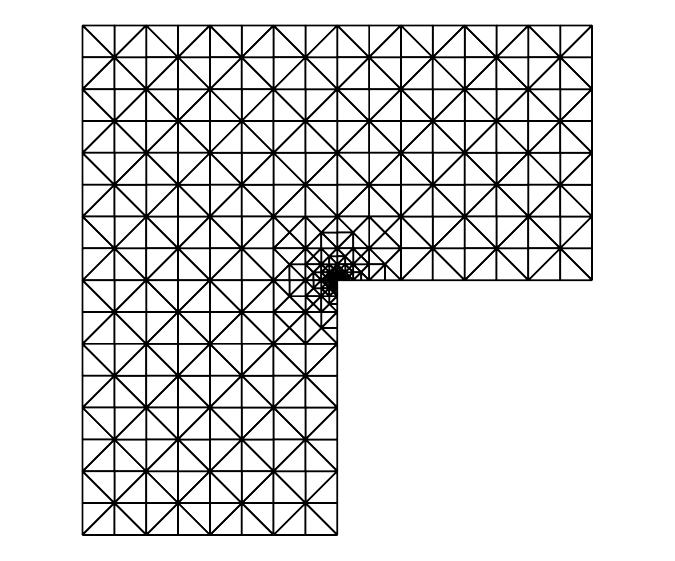}
 \end{center}
 \caption{\label{f:triangulations1}Sample triangulations $\tri_{H}$ (left) and 
 $\tri_{h}$ (right). The shaded area in the left triangulation marks the 
 support of functions in the kernel space $W_h$ (see Subsection~\ref{sss:locality}).}
\end{figure}

\begin{figure}
 \begin{center}
   \includegraphics[height=0.41\textwidth]{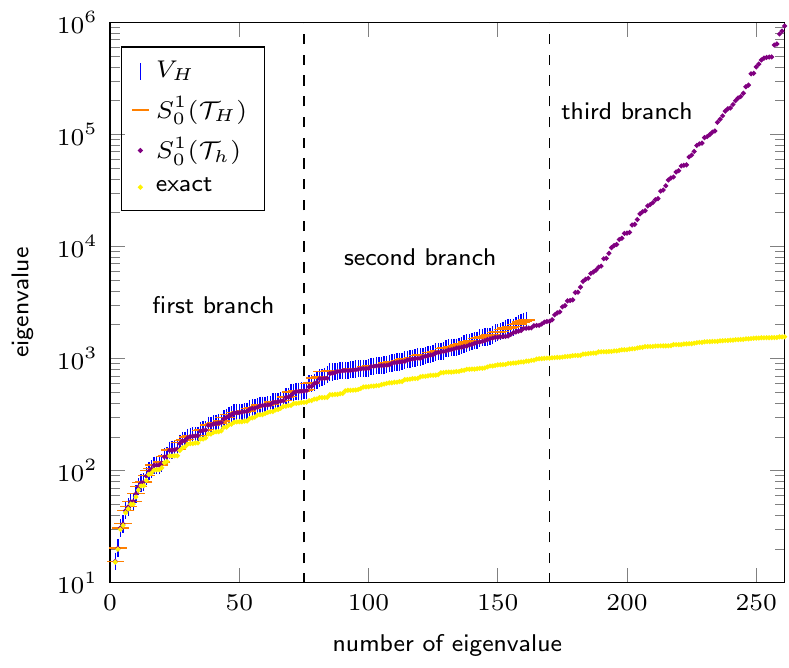}
 \hfill 
   \includegraphics[height=0.41\textwidth]{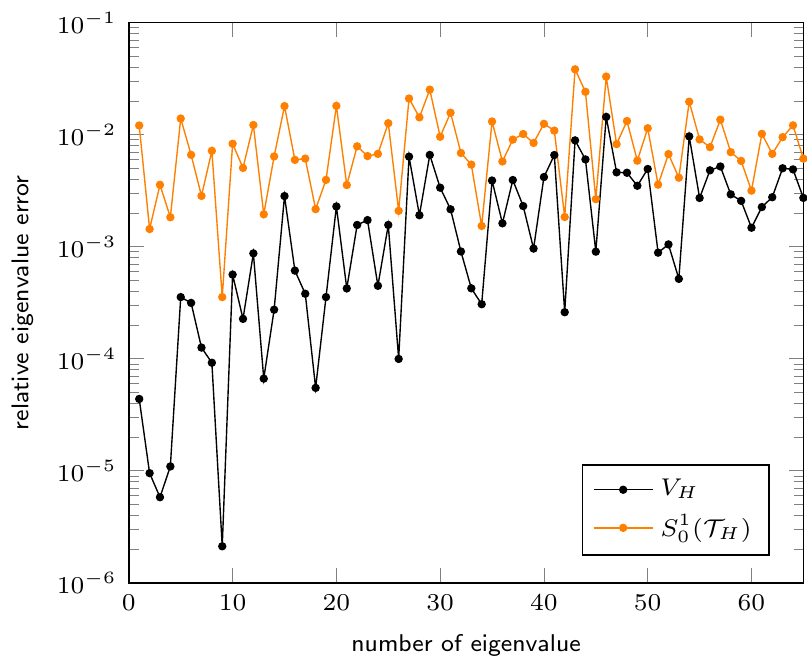}
 \end{center}
 \caption{\label{f:spectrum}
 Spectra of  finite element discretisations  of the 
 Laplacian based on the spaces $S^1_0(\tri_H)$, $S^1_0(\tri_h)$ and the 
 generalised finite element space $V_H$ (left) for the triangulations from 
 Figure~\ref{f:triangulations1}. The exact eigenvalues 
 of the Laplacian with homogeneous Dirichlet boundary condition are depicted 
 for reference. On the right, the relative eigenvalue errors for the 
 approximations in $V_H$ and $S^1_0(\tri_H)$ of the eigenvalues in 
 $S^1_0(\tri_h)$ are plotted for the first branch.}
\end{figure}

We shall finally have a look at the advantageous properties of the reduced 
ansatz space $\corv$ from a different angle. 
Figure~\ref{f:spectrum} shows the eigenvalues related to the $P_1$ finite 
element discretisation of the Laplacian with homogeneous Dirichlet boundary 
condition on the refined triangulation $\tri_h$ depicted in 
Figure~\ref{f:triangulations1}. Essentially, the spectrum consists of three branches indicated by the 
dotted lines. The eigenvalues in the first branch are meaningful 
approximations of the corresponding exact 
eigenvalues of the Laplacian and the corresponding eigenfunctions reflect 
true modes of the operator. The eigenvalues of the second and third branch are  
spurious in the sense that they do not 
approximate Laplacian eigenvalues. The artificial modes of 
the second branch are related to the fact that the finite element space does 
not satisfy $\Delta S^1_0(\tri_h)\subseteq L^2(\Omega)$ 
\cite{Gallistl.Huber.Peterseim:2015}. The artificial eigenvalues of the third 
branch are related to the degrees of freedom introduced by the local mesh 
refinement. The largest of them scales like the reciprocal squared minimal
mesh size which leads to the restrictive CFL condition.  

In a way, this restriction is the 
result of additional flexibility of the 
finite element space introduced through local mesh refinement. The role of 
the reduction process is to eliminate those artificial modes of the third 
branch while preserving the first branch extremely accurately, in particular, 
much more accurately than standard finite elements on the coarse mesh. That this 
is indeed the case is also illustrated in Figure~\ref{f:spectrum} where the 
right plot shows 
that the novel subspace reduction improves the approximation of the first branch by orders of 
magnitude when compared with the standard finite elements $S^1_0(\tri_H)$ of 
the same dimension. 
The largest eigenvalues in $V_H$ and $S^1_0(\tri_H)$ are very close to each 
other, and therefore the CFL condition leads to almost the same restriction 
of the time step size. 
For a rigorous analysis of eigenvalue errors that justifies these observations
we refer to previous works on two-level methods for linear and non-linear 
eigenvalue problems \cite{Henning.Mlqvist.Peterseim:2013,Mlqvist.Peterseim:2015,MalqvistPeterseim2014}.

\subsection{Example for $L^2$ and $H^1$ stable quasi interpolation}
\label{ss:exampleIH}

This subsection gives an example for a projective quasi interpolation operator  
$I_H$ that satisfies~\eqref{e:IHapproxstab} and \eqref{e:IHL2stab}.

Define the space of (possibly discontinuous) piecewise affine functions 
over $\tri_H$ as 
\begin{align*}
 P_1(\tri_H):=\{w\in L^2(\Omega)\mid \forall T\in\tri_H{:}\; w\vert_T\text{ is affine}\}.
\end{align*}
Given $w_H\in P_1(\tri_H)$, define the nodal averaging operator 
$J_1:P_1(\tri_H)\to S^1_0(\tri_H)$ by the averaging over the values 
of adjacent simplices, i.e., 
\begin{align*}
 J_1 w_H (z):= (\mathrm{card}(\{K\in \tri_H\mid z\in K\}))^{-1} 
    \sum_{\substack{K\in\tri_H\\ z\in K}} w_H\vert_K (z)
\end{align*}
for all interior nodes $z$ of $\tri_H$.
This kind of operator is well known in the context of fast 
solvers~\cite{Brenner1996,Oswald93} and in the a~posteriori analysis of 
discontinuous Galerkin methods~\cite{KarakashianPascal2003}.
Let $\Pi_1^{\tri_H}:V\to P_1(\tri_H)$ denote the $L^2$~projection onto 
$P_1(\tri_H)$ and define $I_H:V\to S^1_0(\tri_H)$ by 
\begin{align*}
 I_H v := J_1(\Pi_1^{\tri_H} v)
 \qquad\text{for all }v\in V.
\end{align*}
Then $I_H$ is a projection. The following lemma proves that it  
satisfies the $L^2$~stability~\eqref{e:IHL2stab} and the approximation and 
stability properties~\eqref{e:IHapproxstab}.

\begin{lemma}[stability of $I_H$]\label{l:L2stabIH}
The operator $I_H$ is $H^1$ and $L^2$~stable in the sense that it 
satisfies~\eqref{e:IHapproxstab} and~\eqref{e:IHL2stab}.
\end{lemma}

\begin{proof}
The proof follows from the more general situation in~\cite{ErnGuermond2015},
but is given here for the sake of completeness and self-contained reading.

We first prove the $L^2$~stability of $J_1:P_1(\tri_H)\to S^1_0(\tri_H)$ and 
then conclude the approximation properties and the $H^1$ and $L^2$~stability 
of $I_H$.

Let $w_H\in P_1(\tri_H)$ and $T=\mathrm{conv}\{z_0,\dots,z_d\}\in \tri_H$ 
and let $\lvert T\rvert$ denote the ($d$-dimensional) volume of $T$. 
\new{The definition of $J_1$ involves adjacent simplices of $T$. For such an 
adjacent simplex $K$, let}
$\lambda_\mathrm{max}(T)$ and $\lambda_\mathrm{min}(K)$ denote the maximal 
and the minimal eigenvalue of the mass matrix on the simplex $T$ and $K$.
The definition of $J_1$ then implies 
\begin{align*}
  \|J_1 w_H\|_{L^2(T)}^2
    \leq \lambda_\mathrm{max}(T) \sum_{j=0}^d \lvert(J_1 w_H)(z_j)\rvert^2
    &\leq C \lambda_\mathrm{max}(T) \sum_{j=0}^d
           \sum_{\substack{K\in \tri_H\\z_j\in K}} 
                       \big\lvert w_H\vert_K(z_j)\big\rvert^2\\
    &\leq C \sum_{j=0}^d
           \sum_{\substack{K\in \tri_H\\z_j\in K}} 
             \frac{\lambda_\mathrm{max}(T)}{\lambda_\mathrm{min}(K)}
             \,\|w_H\|_{L^2(K)}^2.
\end{align*}
The shape regularity of $\tri_H$ implies that there exists a generic constant 
$c>0$ with 
$\lvert T\rvert\leq c \lvert K\rvert$ for any $K\in\tri_H$ with 
$K\cap T\neq\emptyset$.
This implies 
\begin{align*}
  \frac{\lambda_\mathrm{max}(T)}{\lambda_\mathrm{min}(K)}
  \leq C.
\end{align*}
Therefore, $J_1$ satisfies a (local) $L^2$~stability, which together with the
fact that the number of overlapping simplices is bounded leads to the 
$L^2$~stability
\begin{align*}
  \|J_1 w_H\|
   \leq C \|w_H\|.
\end{align*}
Since $\Pi_1^{\tri_H}:V\to P_1(\tri_H)$ is the 
$L^2$~projection, this operator is $L^2$~stable. The $L^2$~stability of $J_1$
therefore implies the $L^2$~stability~\eqref{e:IHL2stab} of 
$I_H=J_1\circ \Pi_1^{\tri_H}$.

Let now $v\in V$. A triangle inequality, the fact that 
$\|v-\Pi_1^{\tri_H}v\|\leq \|v-\Pi_0^{\tri_H}v\|$ for the $L^2$~projection 
$\Pi_0^{\tri_H}$ onto piecewise constants and a piecewise Poincar\'e inequality
lead to
\begin{align*}
  H^{-1}\|v-I_H v\|
   &\leq H^{-1}\|v-\Pi_1^{\tri_H} v\| + H^{-1}\|\Pi_1^{\tri_H} v - J_1 \Pi_1^{\tri_H} v\|\\
  & \leq C \|\nabla v\| + H^{-1}\|\Pi_1^{\tri_H} v - J_1 \Pi_1^{\tri_H} v\|.
\end{align*}
Define the set of hyper-surfaces $\mathcal{F}$ and let $[\bullet]_F$ denote 
the jump across a hyper-surface $F\in\mathcal{F}$.
The stability of $J_1$ 
from~\cite[Lemma~4.8]{CarstensenGallistlSchedensack2016} proves
\begin{align*}
  H^{-2}\|\Pi_1^{\tri_H} v - J_1 \Pi_1^{\tri_H} v\|^2
    \leq C\left( \|\nabla \Pi_1^{\tri_H} v\|^2 
       + \sum_{F\in\mathcal{F}} H^{-1} \|[\Pi_1^{\tri_H} v]_F\|_{L^2(F)}^2\right).
\end{align*}
Since $v$ is continuous in the sense of traces, the trace inequality 
from~\cite[Lemma~1.49]{DiPietroErn2012} and the finite overlap 
of patches imply 
\begin{align*}
 \sum_{F\in\mathcal{F}} H^{-1} \|[\Pi_1^{\tri_H} v]_F\|_{L^2(F)}^2
   &= \sum_{F\in\mathcal{F}} H^{-1} \|[v-\Pi_1^{\tri_H} v]_F\|_{L^2(F)}^2\\
   &\leq C\left( H^{-2} \|v-\Pi_1^{\tri_H}v\|^2 + \|\nabla(v-\Pi_1^{\tri_H}v)\|^2\right).
\end{align*}
Again, a piecewise Poincar\'e inequality bounds the first term on the right-hand side 
by $\|\nabla v\|$. An inverse inequality, the $L^2$~stability 
of $\Pi_1^{\tri_H}$ and 
$\Pi_1^{\tri_H} \Pi_0^{\tri_H} v= \Pi_0^{\tri_H} v$ prove for all $T\in\tri_H$ 
that
\begin{align*}
 \|\nabla \Pi_1^{\tri_H} v\|_{L^2(T)} 
   &= \|\nabla (\Pi_1^{\tri_H} v-\Pi_0^{\tri_H} v)\|_{L^2(T)}
   \leq C H^{-1} \|\Pi_1^{\tri_H} (v-\Pi_0^{\tri_H} v)\|_{L^2(T)}\\
   &\leq C H^{-1} \|v-\Pi_0^{\tri_H} v\|_{L^2(T)}
   \leq C \|\nabla v\|_{L^2(T)}.
\end{align*}
The combination of the previous inequalities yield the approximation 
property 
\begin{align}\label{e:approxIHexample}
  H^{-1}\|v-I_H v\|\leq C \|\nabla v\|
\end{align}
of $I_H$.

For the proof of the $H^1$ stability of $I_H$ let $v\in V$. An inverse,  
a triangle and a piecewise Poincar\'e inequality and the approximation 
property~\eqref{e:approxIHexample} lead to 
\begin{align*}
  \|\nabla I_H v\|
    = \|\nabla (I_H v-\Pi_0^{\tri_H} v)\|
    &\leq C H^{-1} \|I_H v-\Pi_0^{\tri_H} v\|\\
    &\leq C H^{-1}\|v - I_H v\| + C H^{-1} \|v-\Pi_0^{\tri_H} v\|
    \leq C \|\nabla v\|.
\end{align*}
This proves~\eqref{e:IHapproxstab} and concludes the proof.
\end{proof}

\section{Application to the wave equation}\label{s:discretisation}

This section defines the leapfrog discretisation of the wave equation based on 
the spatial Galerkin approximation in the reduced space $\corv$ 
in Subsection~\ref{ss:defstabFEM2} and states an error estimate and stability
in Subsection~\ref{ss:errorestimate}. 

Given $f\in L^2(0,T;L^2(\Omega))$,
the wave equation~\eqref{e:strongform} in its weak form seeks 
$u\in L^2(0,T;V)$ with $\dot{u}\in L^2(0,T;L^2(\Omega))$
and $\ddot{u}\in L^2(0,T;H^{-1}(\Omega))$ such that 
for almost all $t\in [0,T]$ and all $v\in V$
\begin{align}\label{e:weakform}
 \left\langle\ddot{u}(t) , v\right\rangle_{H^{-1}(\Omega)\times H^1_0(\Omega)}
   + (\nabla u(t),\nabla v)_{L^2(\Omega)} = (f(t),v)_{L^2(\Omega)}.
\end{align}

\subsection{The leapfrog in the reduced space}\label{ss:defstabFEM2}

Let the time step size $\Delta t$ satisfy the relaxed CFL-condition
\begin{align}\label{e:CFLstab}
 \Delta t < \sqrt{2}/C_\mathrm{inv}(\corv)
\end{align}
and let $N=\lceil T/\Delta t \rceil$ be the number of time steps.
Recall the definition of the space $\corv$ from 
Subsection~\ref{ss:constructioncorv}.
Given approximations $\uH^0\in \corv$ to $u(0)$ and $\uH^1\in\corv$ to 
$u(\Delta t)$,
this method seeks $(\uH^n)_{n=0,\dots,N}$ with 
$\uH^n\in \corv$ such that for all $n=2,\dots,N$ and all
$\vH\in \corv$
\begin{equation}\label{e:defstabFEM2}
\begin{aligned}
 &\left( \frac{\uH^{n+1} - 2 \uH^n + \uH^{n-1}}{(\Delta t)^2},
         \vH\right)_{\hspace{-1ex}L^2(\Omega)} 
   + (\nabla \uH^n, \nabla \vH)_{L^2(\Omega)}
  = \left(f(n \Delta t), \vH\right)_{L^2(\Omega)}.
\end{aligned}
\end{equation}
This is the standard leapfrog time discretisation.
We shall emphasise that the pre-computation of the space $\corv$ needs to 
be done only once. 

Lemma~\ref{l:approxinverse} proves that $C_\mathrm{inv}(V_H)\leq C H^{-1}$
and, hence, 
the CFL condition~\eqref{e:CFLstab} states that the time step size is 
in the range of the mesh size of the (coarse) quasi-uniform triangulation $\tri_H$.
In the presence of singularities, this is a much weaker condition 
compared with the CFL condition~\eqref{e:CFLintro} for the space 
$S^1_0(\tri_h)$.

One of the fundamental properties of the leapfrog scheme is the conservation 
of energy in the following sense.
Given $(\vH^n)_{n=0,\dots,N}$ with $\vH^n\in\corv$ for $n=0,\dots,N$, define the discrete 
time derivative by
\begin{align*}
 \dot{v}_H^{n+1/2}=\frac{\vH^{n+1}-\vH^n}{\Delta t}
\end{align*}
for $n=1,\dots,N-1$ and
define the discrete energy as 
\begin{align*}
 \energy^{n+1/2}((\vH^k)_{k=0,\dots,N})
    := \tfrac{1}{2}\big(\|\dot{v}^{n+1/2}_H\|^2 + (\nabla \vH^n,\nabla\vH^{n+1})\big).
\end{align*}
Then the discrete energy of the solution $(\uH^k)_{k=0,\dots,N}$ 
of~\eqref{e:defstabFEM2} is conserved in the sense that 
\begin{align*}
 \energy^{n+1/2}((\uH^k)_{k=0,\dots,N})
   = \energy^{n-1/2}((\uH^k)_{k=0,\dots,N})
      + \tfrac{1}{2}\Delta t \big(f(t_n),\dot{u}^{n+1/2}_H+ \dot{u}^{n-1/2}_H\big)_{L^2(\Omega)}.
\end{align*}

\subsection{Stability and error estimates}\label{ss:errorestimate}

The following theorem estimates the difference between the discrete solution 
of~\eqref{e:defstabFEM2} and the exact solution $u$ of~\eqref{e:weakform}.
Let $z_H\in L^2(0,T;\corv)$ denote the auxiliary semi-discrete solution, i.e.,  
$\dot{z}_H\in L^2(0,T;\corv)$,  $\ddot{z}_H\in L^2(0,T;\corv)$ and $z_H$ solves
\begin{align*}
 \left\langle\ddot{z}_H(t) , \vH\right\rangle_{H^{-1}(\Omega)\times H^1_0(\Omega)}
   + (\nabla z_H(t),\nabla \vH)_{L^2(\Omega)} = (f(t),\vH)_{L^2(\Omega)}
   \quad\text{for all }\vH\in \corv
\end{align*}
for almost all $t\in[0,T]$
with initial conditions $z_H(0)=\uH^0$ and $\dot{z}_H(0)=z_H^0$ for some 
$z_H^0$.
As usual, the error is split in the time discretisation error $(E^n)_{n=0,\dots,N}$ 
defined by $E^n:=\uH^n - z_H(n\Delta t)$ and the space 
discretisation error $z_H(n\Delta t) - u(n\Delta t)=
z_H(n\Delta t)-\Pi_{\corv} u(n\Delta t)-\varepsilon(n\Delta t)$ with 
the best-approximation error
$\varepsilon(t):=u(t)-\Pi_{\corv} u(t)$, where 
$\Pi_{\corv} u(t)$ denotes the orthogonal projection of $u(t)$
onto $\corv$ with respect to the bilinear form 
$(\nabla\bullet,\nabla\bullet)_{L^2(\Omega)}$.

The proof of the following theorem is based on the conservation of the 
discrete energy from Subsection~\ref{ss:defstabFEM2} and follows as for the 
standard leapfrog scheme (see~\cite{Christiansen2009,Joly2003}) and is 
therefore dropped. 

\begin{theorem}[error estimate for reduced FEM]\label{t:totalerrorest}
If the relaxed CFL condition~\eqref{e:CFLstab} 
is satisfied, then~\eqref{e:defstabFEM2} is stable in the sense that 
\begin{align*}
  \|\dot{u}_H^{n+1/2}\| + \|\nabla u_H^{n+1}\|
  \leq C \left(\|\dot{u}_H^{1/2}\| + \|\nabla u_H^0\|
     + \|\nabla u_H^1\| + \sum_{k=2}^n \Delta t\|f(k\Delta t)\|\right).
\end{align*}
Furthermore, 
if $\ddot{u}\in L^1(0,T;L^2(\Omega))$ and $\ddot{z}_H\in C(0,T;V_H)$,
then it holds with $t_n=n\Delta t$ that
\begin{equation}\label{e:totalerrorest}
\begin{aligned}
 &\left\|\frac{(\uH^n-u(t_n))
          -(\uH^{n-1}+u(t_{n-1}))}{\Delta t}\right\|
   + \|\nabla (\uH^n-u(t_n))\|\\
 &\;\leq C\left(\|\dot{E}^{1/2}\| + \|\nabla E^1\|
    + \left\|\dot{z}_H(0)-\Pi_{\corv} \dot{u}(0)\right\|
      + \|\nabla (z_H(0)-\Pi_{\corv} u(0)\| \right.\\
  &\qquad \qquad  \;\left.
      + \left\|\frac{\varepsilon(t_n)-\varepsilon(t_{n-1})}{\Delta t}\right\|
      + \|\nabla \varepsilon(t_n)\|
      + \int_0^{t_n}
         \left\|\ddot{\varepsilon}(s)\right\| \,ds
      \right.\\
   &\qquad\qquad  \;    \left.   
    + \sum_{j=1}^n \Delta t
       \left\|\frac{z_H(t_{j+1}) - 2 z_H(t_j)
                              + z_H(t_{j-1})}{(\Delta t)^{2}}
           - \ddot{z}_H(t_j)\right\|\right).\\
\end{aligned}
\end{equation}
\end{theorem}

Note that the fifth to seventh term on the right-hand side 
of~\eqref{e:totalerrorest} only contain the best-approximation error 
$\varepsilon$ of $u$ in $\corv$. Therefore, Lemma~\ref{l:approxinverse}
can be applied. 
If $u\in C^2(0,T;L^2(\Omega))$ and $f\in C(0,T;L^2(\Omega))$, then 
$\Delta u=\ddot{u}-f\in C^0(0,T;L^2(\Omega))$, and the term 
$\|\Delta u\|_{L^2(\Omega)}$ in the right-hand side of~\eqref{e:approxcorv}
is bounded.
Therefore, under the additional (standard) regularity assumptions
$u\in C^4(0,T;L^2(\Omega))$ and $f\in C^2(0,T;L^2(\Omega))$, the fifth 
to seventh term 
can be bounded by $H$. 
With the regularity assumption $z_H\in C^4(0,T;L^2(\Omega))$, the last 
term on the right-hand side of~\eqref{e:totalerrorest} converges as $(\Delta t)^2$.
For suitable initial conditions, this leads to a convergence rate of 
the approximation~\eqref{e:defstabFEM2} of $H+(\Delta t)^2$.

\section{Practical aspects and possible generalisations}\label{s:praxis}

This section is concerned with practical aspects of the computation of 
$(\uH^n)_{n=0,\dots,N}$ from~\eqref{e:defstabFEM2}. Subsection~\ref{ss:sparsity} discusses the 
sparsity properties of the stiffness and mass matrix associated with the 
reduced space $\corv$. Since the computation of $\corv$ has to be done only 
once, the sparsity properties serve as measure for the overall complexity.
Subsection~\ref{ss:preconditioner}
shows that the inverse diagonal is an optimal preconditioner for the mass 
matrix of the reduced space.
Subsection~\ref{ss:dG} concludes this section 
with a generalisation to discontinuous Galerkin FEMs.

\subsection{Sparsity of the reduced space}\label{ss:sparsity}

In contrast to standard finite element spaces, basis functions of $\corv$ 
are not a~priori known, but can be defined by the canonical choice 
$\lambda_z-\cor\lambda_z$ for the standard nodal basis functions $\lambda_z$ 
of $S^1_0(\tri_H)$. However, problem~\eqref{e:corprob} for the computation of 
$\cor\lambda_z$ is formulated on the whole domain $\Omega$, and might lead 
to global basis functions and dense stiffness and mass matrices.
This subsection identifies cases where sparsity is automatically preserved
(Subsection~\ref{sss:locality}) and shows how to achieve sparse approximations
in the general case (Subsection~\ref{sss:localisation}).

\subsubsection{Sparsity on locally adaptive meshes}\label{sss:locality}

We consider the case that the triangulation $\tri_h$ refines $\tri_H$ 
only locally in the sense that only a small number of coarse elements is 
actually refined, for an example see Figure~\ref{f:triangulations1}.
In this case the corrector problem \eqref{e:corprob} reduces 
to a local one because the kernel space $W_h$ vanishes outside 
of (one layer around) the refined region (see the following proposition). 
Therefore, $\cor \lambda_z = 0$ for all basis functions $\lambda_z$ for 
nodes $z$ of $\tri_H$ with the property that the two-layer patch 
$$\Omega_z:=\{x\in\Omega\mid \exists T,T'\in\tri_H\text{ such that }
z\in T, x\in T'\text{ and }T\cap T'\neq\emptyset\}$$ lies in the non-refined 
region, $\Omega_z\subseteq \bigcup(\tri_H\cap\tri_h)$. 
In the example of Figure~\ref{f:triangulations1}, all nodes $z$
for which $\cor \lambda_z\neq 0$ are highlighted as well as 
the union of the supports 
of functions in $W_h$.

\begin{proposition}[locality of problems~\eqref{e:corprob}]\label{p:localcorrectors}
Assume that $I_H$ satisfies the local $L^2$~stability
\begin{align*}
 \|I_H v\|_{L^2(T)}
   \leq C \|v\|_{L^2(\Omega_T)}
   \qquad\text{for all }v\in V
\end{align*}
for all $T\in\tri_H$ and $\Omega_T=\bigcup \{K\in\tri_H\mid K\cap T\neq\emptyset\}$. 
Let $w_h\in W_h$. Then 
$w_h\vert_{\widetilde{\Omega}}=0$ for 
$\widetilde{\Omega}=\bigcup \{T\in\tri_H\mid \forall K\in \tri_H\text{ with }
K\cap T\neq\emptyset\text{ it holds }K\in\tri_H\cap\tri_h\}$.
\end{proposition}

\begin{proof}
Let $\mathcal{N}$ denote the set of nodes in $\tri_h$ that are not in 
$\tri_H$ and define $w_y:=\lambda_y - I_H\lambda_y$ for $y\in \mathcal{N}$. 
We want to show that the functions $w_y$ for $y\in \mathcal{N}$ are 
linear independent.
%Since $I_H$ is a projection, $w_y\in W_h$. 
Let $\alpha_y\in\R$ such that 
\begin{align*}
 \sum_{y\in\mathcal{N}} \alpha_y w_y=0.
\end{align*}
On the one hand, the definition of $w_y$ leads to
$\sum_{y\in\mathcal{N}} \alpha_y \lambda_y 
= \sum_{y\in\mathcal{N}}\alpha_y I_H\lambda_y\in S^1_0(\tri_H)$, 
i.e., the function $\sum_{y\in\mathcal{N}} \alpha_y \lambda_y $ is 
piecewise affine on the triangulation $\tri_H$.
On the other hand, the functions $\lambda_y$ vanish at all nodes in $\tri_H$.
This implies that the functions $w_y$ are 
linear independent. A dimension argument proves that they form a basis of $W_h$.
The local $L^2$~stability implies that $w_y$ has 
the local support $\{x\in\Omega\mid \exists T,T'\in \tri_H\text{ such that }
x\in T, y\in T'\text{ and }T\cap T'\neq \emptyset\}$. 
\end{proof}

Proposition~\ref{p:localcorrectors} implies that the number of additional non-zero 
entries in the mass and stiffness matrix depends only
on the number of triangles of $\tri_H$ that are refined in $\tri_h$.

\subsubsection{Sparsification on graded meshes}\label{sss:localisation}

In this subsection we consider an arbitrary refinement $\tri_h$ of $\tri_H$.
Given a coarse nodal basis function $\lambda_z$, it was shown 
in \cite{MalqvistPeterseim2014_localization} that 
$\cor\lambda_z$ decays exponentially fast outside of the support of 
$\lambda_z$ (see \cite{Peterseim:2015} for illustrations). This decay allows 
the truncation of the computational domain 
for~\eqref{e:corprob} to local subdomains of diameter $m H$, roughly, 
where $m$ denotes a new discretisation parameter, namely the localisation 
or oversampling parameter. The obvious way would be to simply replace the 
global domain $\Omega$ in the computation of $\cor\lambda_z$ with suitable 
neighbourhoods of the nodes $z$. This procedure was used 
in~\cite{MalqvistPeterseim2014_localization}. However, it turned out that 
it is advantageous to consider the following slightly more involved technique 
based on element 
correctors~\cite{HenningPeterseim2013,HenningMorgensternPeterseim2014}.
Define the $m$-th order patch 
\begin{align*}
 \Omega_{T,m}:=\bigcup\left\{ K\in\tri_H\left| 
     \begin{array}{l}
      \exists K_0,\dots,K_m\in \tri_H \text{ with }K_0=T, K_m=K \\ 
         \text{ and }K_j\cap K_{j+1}\neq\emptyset
         \text{ for all }j=0,\dots,m-1
     \end{array}
       \right.\right\}.
\end{align*}
We introduce corresponding truncated function spaces
\begin{align*}
  W_h(\Omega_{T,m})
    :=\{w_h\in W_h\mid \mathrm{supp}(w_h)\subseteq \Omega_{T,m}\}.
\end{align*}
Given any coarse nodal basis function $\lambda_z\in S^1_0(\tri_H)$, let 
$\cor_{T,m} \lambda_z\in W_h(\Omega_{T,m})$ solve the localised element 
problem 
\begin{align*}
  (\nabla w_h,\nabla \cor_{T,m} \lambda_z)_{L^2(\Omega)}
   = \int_T \nabla w_h\cdot \nabla \lambda_z
   \quad\text{for all }w_h\in W_h(\Omega_{T,m})
\end{align*}
and define $\cor_m \lambda_z:=\sum_{T\in\tri_H} \cor_{T,m} \lambda_z$.
Note that we impose homogeneous Dirichlet boundary conditions on the 
artificial boundary of the patch $\Omega_{T,m}$ which is well justified 
by the fast decay. 
Under the assumption that $\tri_h$ is shape regular and that $I_H$ is a local 
operator (as the one introduced in Subsection~\ref{ss:exampleIH})  
it is proved in 
\cite{HenningPeterseim2013,MalqvistPeterseim2014_localization,HenningMorgensternPeterseim2014} that this leads to 
the existence of constants $C>0$ and $\beta>0$ such that 
\begin{align}\label{e:decay}
  \|\nabla(\cor \lambda_z - \cor_m \lambda_z)\| 
     \leq C \exp(-\beta m) \|\nabla \lambda_z\|.
\end{align}
This justifies the utilisation of 
\begin{align*}
  \corv^{(m)}:=\mathrm{span}\{\lambda_z-\cor_m\lambda_z
    \mid z\text{ interior node of }\tri_H\}
\end{align*}
as an approximation to $\corv$. Due to the exponential decay \eqref{e:decay}, 
the choice of $m\approx\lvert\log(H)\rvert$ ensures that this perturbation does not 
affect the advantageous approximation properties of $V_H$. 

For the construction of the basis, $O(H^{-d})$ problems have to be solved. 
Each of these problems consists of $O((\log(H)H/h_\mathrm{min})^d)$ 
degrees of freedoms in the worst case, depending on the grading of the fine triangulation.
These costs are 
offline costs in the sense that the basis has to be constructed in the 
beginning only and does not depend on the number of time steps. 
It does only depend on the coarse and the fine mesh.
The non-zero entries in the mass and stiffness matrix amount to 
$O((2\log(H)/H)^{d})$.

\subsection{Diagonal preconditioning of the mass matrix}\label{ss:preconditioner}

This subsection proves that the inverse of the diagonal of the mass matrix is 
a suitable preconditioner for it. Although this is shown for $\corv$ with 
basis functions $\lambda_z-\cor\lambda_z$, the arguments and therefore 
also the result hold as well for the perturbed spaces $\corv^{(m)}$ 
of Subsection~\ref{sss:localisation} spanned by the local basis functions 
$\lambda_z-\cor_m\lambda_z$.

Define $D:=\mathrm{dim}(S^1_0(\tri_H))$. 
Let $(\Lambda_k)_{k=1,\dots,D}$ denote the basis of $S^1_0(\tri_H)$ 
consisting of the standard nodal basis functions. Then 
$(\lambda_k)_{k=1,\dots,D}$ with 
$\lambda_k=(1-\cor)\Lambda_k$ for $\cor$ from~\eqref{e:corprob} defines a 
basis of $\corv$.
Let $\mathcal{M}$ and $M$ denote the mass matrices with respect to 
$(\Lambda_k)_{k=1,\dots,D}$ and 
$(\lambda_k)_{k=1,\dots,D}$.
Let $x\in\R^{D}$ and let 
\begin{align*}
 U_H:=\sum_{k=1}^{D}x_k\Lambda_k\in S^1_0(\tri_H)
 \qquad\text{and}\qquad 
 \uH:= \sum_{k=1}^{D}x_k\lambda_k
=(1-\cor) U_H\in\corv
\end{align*}
denote the corresponding functions in $S^1_0(\tri_H)$ and 
$\corv$.
Then, 
\begin{align*}
  x^\top M x
    = \|\uH\|^2 = \|(1-\cor)U_H\|^2
    \leq (\|U_H\| + \|\cor U_H\|)^2.
\end{align*}
Since $\cor U_H\in W_h$, the approximation properties~\eqref{e:IHapproxstab}
and the inverse inequality~\eqref{e:inverseineq} imply
\begin{align*}
   \|U_H\| +  \|\cor U_H\|
     &\leq \|U_H\| +C_{I_H}^{(1)} H\|\nabla U_H\|
     \leq \left(1+ C_{I_H}^{(1)} H C_\mathrm{inv}(S^1_0(\tri_H))\right) \|U_H\|\\
     &= \left(1+C_{I_H}^{(1)} H C_\mathrm{inv}(S^1_0(\tri_H))\right) \sqrt{x^\top \mathcal{M} x}.
\end{align*}
On the other hand, since $U_H=I_H \uH$, the 
$L^2$~stability~\eqref{e:IHL2stab} leads to 
\begin{align*}
  \sqrt{x^\top \mathcal{M} x} = \|U_H\|
   \leq C_{I_H}^{(0)} \|\uH\| = C_{I_H}^{(0)}\, \sqrt{x^\top M x}.
\end{align*}
Given $A,B\in\R$, let $A\approx B$ abbreviate that there exist generic 
constants $C_1>0$, $C_2>0$ independent of the mesh size, such that 
$A\leq C_1 B\leq C_2 A$. Since $C_\mathrm{inv}(S^1_0(\tri_H))\approx H^{-1}$, the 
above result reads 
\begin{align*}
  x^\top M x \approx x^\top \mathcal{M} x.
\end{align*}
Since this result also holds for the unit vectors 
$e_j\in \R^{D}$, this implies 
\begin{align*}
  x^\top \mathrm{diag}(\mathcal{M}) x
    = \sum_{j=1}^{D} x_j^2 e_j^\top \mathcal{M} e_j
    \approx \sum_{j=1}^{D} x_j^2 e_j^\top M e_j
    = x^\top\mathrm{diag}(M) x.
\end{align*}
Since $x^\top \mathrm{diag}(\mathcal{M}) x \approx x^\top \mathcal{M} x$ 
\cite{Wathen01101987}, it follows 
\begin{align*}
  x^\top\mathrm{diag}(M) x
  \approx x^\top\mathrm{diag}(\mathcal{M}) x
  \approx  x^\top \mathcal{M} x
  \approx x^\top M x.
\end{align*}
Therefore, $(\mathrm{diag}(M))^{-1}$ is a suitable preconditioner for 
$M$ and the application of $M^{-1}$ may be replaced with a few iterations of 
the preconditioned conjugate gradient method.

\subsection{Generalisation to other FEMs}\label{ss:dG}

The reduction approach is not at all restricted to linear conforming finite 
elements. The generalisation to many non-standard schemes is possible. 
If one considers, e.g., a discontinuous Galerkin discretisation 
instead of a FEM approximation,
then $I_H=\Pi_\mathrm{dG}$ the $L^2$~projection onto the discontinuous 
Galerkin space satisfies the approximation properties~\eqref{e:IHapproxstab}
and the $L^2$~stability~\eqref{e:IHL2stab} with $\|\nabla \bullet\|$
replaced by the dG norm 
\begin{align*} 
\sqrt{\sum_{T\in\tri_h}\|\nabla \bullet\|_{L^2(T)}^2 
+\sum_{F\in\mathcal{F}} \frac{\sigma}{\mathrm{diam}(F)} \|[\bullet]_F\|_{L^2(F)}^2},
\end{align*}
where $\mathcal{F}$ denotes the set of hyper-surfaces of $\tri_h$ 
(e.g., the set of edges for $d=2$ and the set of faces for $d=3$), $[\bullet]_F$ 
denotes the jump across a hyper-surfaces $F$ and $\sigma$ is some penalty parameter.
Furthermore, Lemma~\ref{l:approxinverse} holds equally with this choice of 
quasi-interpolation. 
\new{In the context of numerical homogenisation, the reduced space for 
discontinuous Galerkin discretisations was utilised  
in~\cite{Elfverson.Georgoulis.Mlqvist.ea:2012}.}

% Higher order elements are also possible in principal. The design of
% a suitable interpolation operator $I_H$ with additional properties 
% has to ensure that the approximation properties of $\corv$ corresponding 
% to Lemma~\ref{l:approxinverse} yield the correct convergence rate 
% (for smooth right-hand sides $\Delta u$).

\new{Higher order elements are also possible in principal, if $\Delta u$ is 
sufficiently smooth. The design of
a suitable interpolation operator $I_H$ with additional properties 
is crucial: 
Additional orthogonality properties
have to ensure that the term $(-\Delta u,e_h)_{L^2(\Omega)}$ with 
$e_h\in \mathrm{ker}(I_H\vert_{S^1_0(\tri_h)})$ in the proof 
of Lemma~\ref{l:approxinverse} converges with the correct rate.}

\section{Numerical experiment}\label{s:numerics}

\begin{figure}
 \begin{center}
   \hspace{1cm}
   \includegraphics[width=0.35\textwidth]{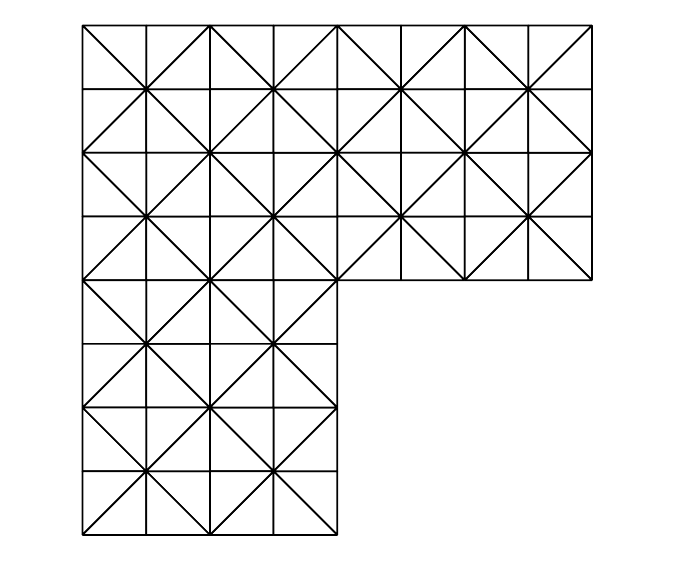}
   \hfill 
   \includegraphics[width=0.35\textwidth]{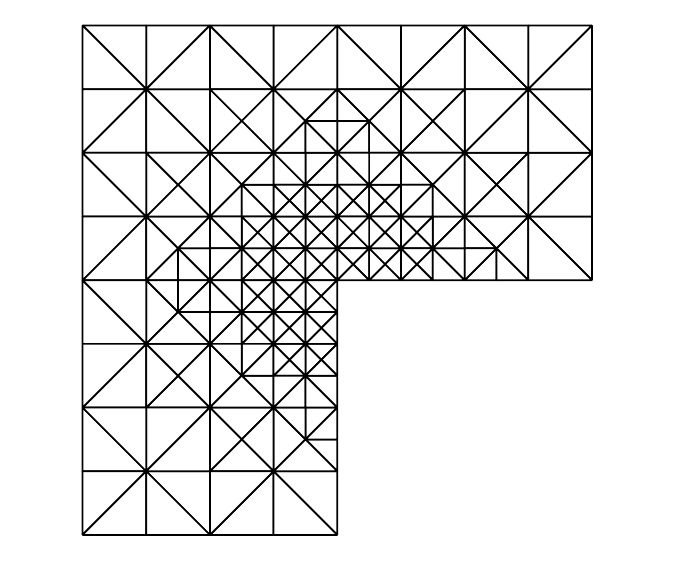}
   \hspace{1cm}
 \end{center}
 \caption{\label{f:triang}Triangulations $\tri_{H,1}$ and $\tri_{h,1}$ for the 
 numerical experiment from Section~\ref{s:numerics}.}
\end{figure}

In this example we consider the wave equation~\eqref{e:strongform} on 
the spatial L-shaped domain 
$\Omega:=(-1,1)^2\setminus ([0,1]\times [-1,0])\subseteq\R^2$ for the 
time interval $[0,0.5]$ with 
inhomogeneous Dirichlet boundary conditions $u\vert_{\partial\Omega}$  
and right-hand side $f$ and initial conditions $u(0)$ and 
$\dot{u}(0)$ given by the exact singular solution 
\begin{align*}
    u(t,x)=\sin(\pi t)\, (r(x))^{2/3}\,\sin(2k\theta(x)/3)
\end{align*}
in polar coordinates $(r,\theta)$. 
The discretisation~\eqref{e:defstabFEM2} can naturally be generalised to 
this case of inhomogeneous Dirichlet boundary conditions.

We consider a sequence of uniform triangulations 
$(\tri_{H,\ell})_{\ell=1,\dots,10}$, such that $\tri_{H,\ell+1}$ is created 
from $\tri_{H,\ell}$ by a proper bisection of every triangle (i.e., the 
longest edge in a triangle is bisected).
The sequence $(\tri_{h,\ell})_{\ell=1,\dots,10}$ consists of triangulations 
such that $\tri_{h,\ell}$ is a refinement of $\tri_{H,\ell}$ created 
similar as in the algorithm \texttt{threshold} from~\cite{GaspozMorin2008}, i.e., 
$\tri_{h,\ell}$ is graded towards the re-entrant corner $(0,0)$. 
The first triangulations $\tri_{H,1}$ and $\tri_{h,1}$ are depicted in 
Figure~\ref{f:triang}.
These triangulations define the finite element spaces $S^1_0(\tri_H)$, 
$S^1_0(\tri_h)$. Let $I_{H,\ell}:=J_1 \circ \Pi_1^{\tri_{H,\ell}}$ as in 
Subsection~\ref{ss:exampleIH}. This defines $\corv$.
The time step size $\Delta t$ for the standard leapfrog FEM on $S^1_0(\tri_H)$ 
and the reduced FEM from~\eqref{e:defstabFEM2}
(resp.\ $\Delta t_h$ for the standard leapfrog FEM on $S^1_0(\tri_h)$) is defined 
by 
\begin{align}\label{e:numdeftimestepsize}
  \Delta t :=\sqrt{2}/ C_\mathrm{inv}(S^1_0(\tri_H))
  \qquad (\text{resp.\ }\;
  \Delta t_h := \sqrt{2} / C_\mathrm{inv}(S^1_0(\tri_h))).
\end{align}
These time step sizes are summarised in Table~\ref{tab:timestepsizes} for the 
triangulations $(\tri_{H,\ell})_{\ell=1,\dots,10}$ and 
$(\tri_{h,\ell})_{\ell=1,\dots,10}$. 
\begin{table}
\begin{center}
{\small
%  \begin{tabular}{l|cccccccccc}
%    $\ell$ & 1&2&3&4&5&6&7&8&9&10\\
%    \hline 
%    $\Delta t$&7.8e-2 &5.5e-2 &3.9e-2 &2.7e-2 &1.9e-2 &1.3e-2 
%          &9.8e-3 &6.9e-3 &4.9e-3 &3.4e-3\\
%    $\Delta t_h$& 3.0e-4& 5.4e-5& 4.7e-6 & 8.4e-7
% %    \multicolumn{2}{c}{}\\
% %    $\ell$ & 9&10\\
% %    \cline{1-3}
% %    %\hline
% %    $\Delta t$  &4.9e-3 &3.4e-3
%  \end{tabular}
 \begin{tabular}{l|ccccccccccc}
   $\ell$ & 1&2&3&4&5&6&7&8&9&10&11\\
   \hline 
   $\Delta t$&1.9e-2 &1.3e-2 &9.8e-3 &6.9e-3 &4.9e-3 &3.4e-3 
         &2.4e-3 & 1.7e-3 &1.2e-3 &8.6e-4 & 6.1e-4\\
   $\Delta t_h$& 6.5e-3& 3.4e-3& 1.6e-3 & 5.7e-4 & 4.0e-4 & 2.1e-4 & 1.0e-4 & 
       3.7e-5 & 2.6e-5 %& 9.5e-6
%    \multicolumn{2}{c}{}\\
%    $\ell$ & 9&10\\
%    \cline{1-3}
%    %\hline
%    $\Delta t$  &4.9e-3 &3.4e-3
 \end{tabular}
 }
\end{center}
\caption{\label{tab:timestepsizes}Time-step sizes for uniform triangulations 
        $\tri_H$ and refined triangulations $\tri_h$ defined 
        by~\eqref{e:numdeftimestepsize}.
        The small time-step sizes $\Delta t_h$ limit the feasible computations
        of a solution of the standard leapfrog FEM on $S^1_0(\tri_h)$
        to the first nine levels.}
\end{table}
While $\Delta t$ is only moderately small for all considered triangulations,
the fine time step-size $\Delta t_h$ decreases with higher rate, such that 
50 times more time steps are needed for the leapfrog on $S^1_0(\tri_h)$
compared with $S^1_0(\tri_H)$ for $\ell=9$.
The approximation $\corv^{(m)}$ of $\corv$ from Subsection~\ref{sss:localisation} 
is employed in the numerical computations with $m=\lceil-0.5\log_2(H)\rceil$, 
which implies $1\leq m \leq 3$ for the performed computations.
The inversions of the mass matrices $M$ are performed with the preconditioned 
conjugate gradients method with preconditioner $\mathrm{diag}(M)^{-1}$. 
\begin{figure}
 \begin{center}
   \includegraphics[width=0.9\textwidth]{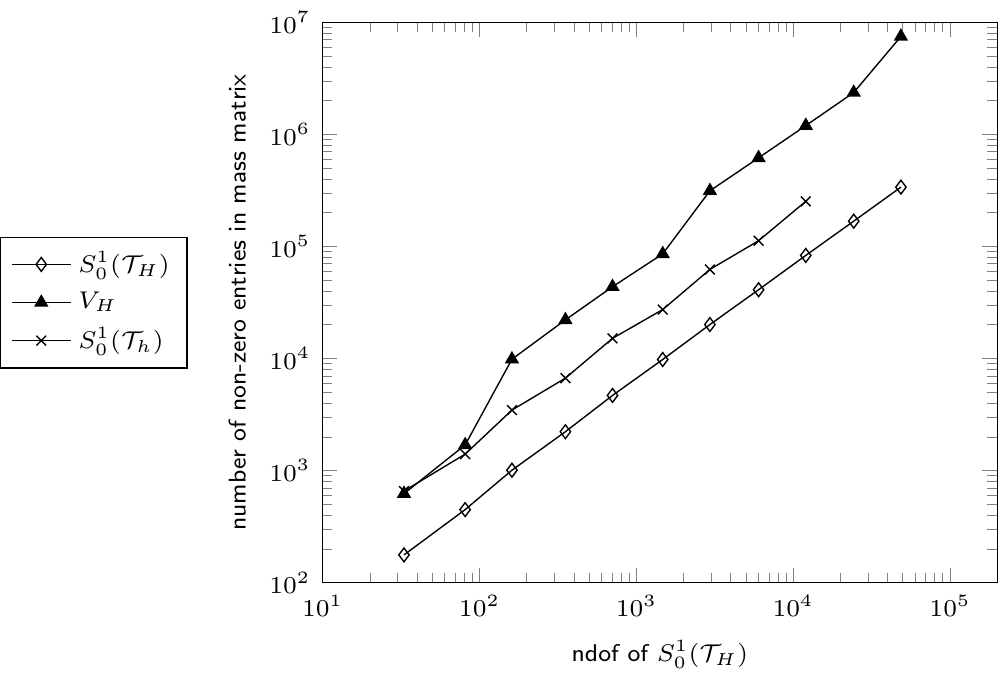}
 \end{center}
 \caption{\label{f:nonzero}Number of non-zero entries in the mass matrices.}
\end{figure}
The number of non-zero entries in the mass matrices are plotted in 
Figure~\ref{f:nonzero} and serve as a measure of the complexity.
The errors 
\begin{align}\label{e:numErrorDef}
  \sum_{k=1}^{N} \Delta t \left\|\nabla(u(k\Delta t)-U^k)\right\|
\end{align}
for the discrete solution $(U^k)_{k=1,\dots,N}$ of the reduced FEM 
of~\eqref{e:defstabFEM2}, of the standard leapfrog FEM on $S^1_0(\tri_H)$ (i.e., 
\eqref{e:defstabFEM2} with $\corv$ replaced by the coarse FEM space $S^1_0(\tri_H)$), 
and of the standard leapfrog FEM on $S^1_0(\tri_h)$ (with $\Delta t$ replaced with the 
fine time step size $\Delta t_h$ and $N$ replaced by 
$\lceil T/\Delta t_h\rceil$) serve as approximations for the error 
in $L^2(0,T;H^1_0(\Omega))$ and are plotted in Figure~\ref{f:numErrorsexL}
against the number of degrees of freedom (ndof) in $S^1_0(\tri_H)$ (which 
equals the number of degrees of freedom in $\corv$).
\begin{figure}
  \begin{center}
   \includegraphics[width=\textwidth]{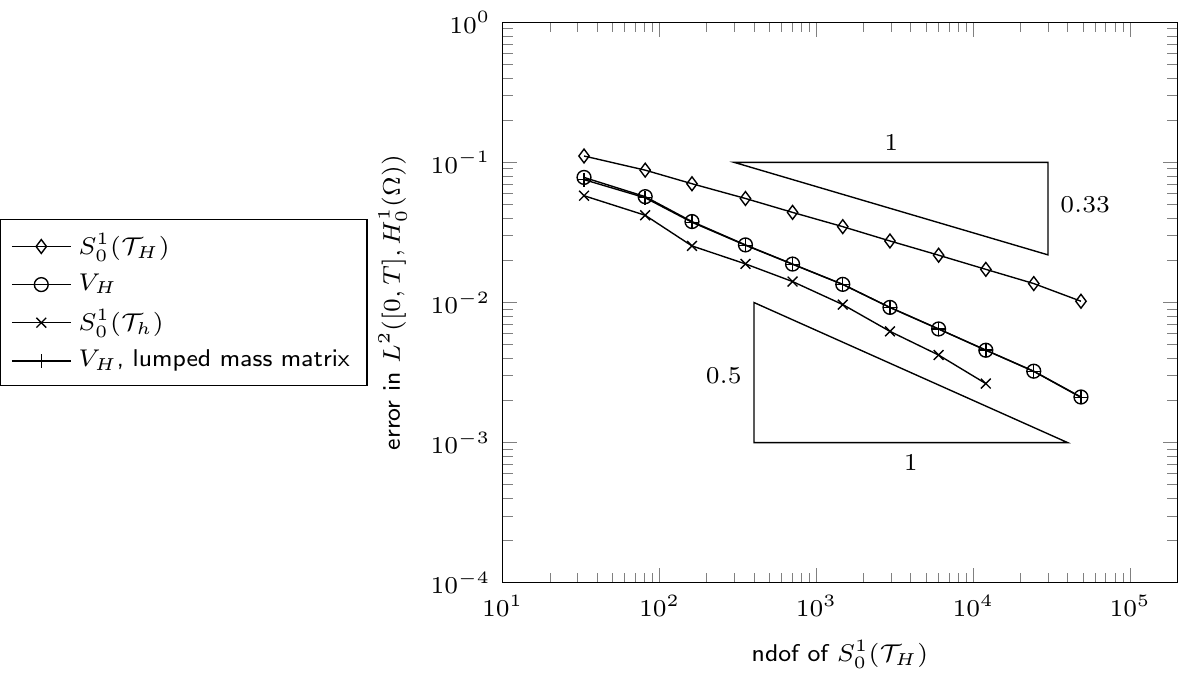}
  \end{center}
  \caption{\label{f:numErrorsexL}Errors~\eqref{e:numErrorDef} for the 
      example from Section~\ref{s:numerics}}
\end{figure}
The error for the standard leapfrog FEM for $S^1_0(\tri_H)$, i.e., on uniform 
triangulations, shows a suboptimal 
convergence rate of $\mathrm{ndof}^{1/3}\approx H^{2/3}$, while the 
approximation with \eqref{e:defstabFEM2} and the standard leapfrog FEM on 
$S^1_0(\tri_h)$ yield the optimal convergence rate of $\mathrm{ndof}^{1/2}$
as predicted by Theorem~\ref{t:totalerrorest}.
\new{Figure~\ref{f:numErrorsexL} also contains the errors of the leapfrog 
FEM for $\corv$ with the mass matrix replaced by the lumped mass matrix, i.e., the 
diagonal matrix whose entry $M_{jj}$ equals the sum of the row entries 
$(M_{jk})_{k=1,\dots,\mathrm{dim}(\corv)}$. The error shows the same behaviour 
as without mass lumping.}

\smallskip
Due to the small time step sizes, 
implicit schemes were previously employed for numerical experiments 
on adaptive meshes \cite{MuellerSchwab2015}, which require 
the expensive inversion of the stiffness matrix in every time step.
The reduced space $\corv$ overcomes the restrictive CFL condition and 
turns the leapfrog into a practicable scheme: 
At the expense of a moderately increased complexity,
the optimal convergence rate is recovered, but with the same time step sizes 
as for uniform meshes.
%
% the feasible computations of the 
% standard leapfrog on
% $S^1_0(\tri_h)$ lead to an accuracy that lies in the range of the accuracy 
% of the feasibly computable leapfrog on $S^1_0(\tri_{H,10})$.
% Adaptive mesh-refinement therefore 
% seems to be not advantageous in practical computations for the explicit 
% leapfrog;
% the effect of the small time step size due to the CFL condition outbalances
% the better convergence rate.
% %
% Therefore, implicit schemes were previously employed for numerical experiments 
% on adaptive meshes \cite{MuellerSchwab2015}, which require 
% the expensive inversion of the stiffness matrix in every time step.
% %
% The reduced space $\corv$ overcomes this restrictive CFL condition and 
% turns the leapfrog into a practicable scheme: 
% At the expense of a moderately increased complexity,
% the optimal convergence rate is recovered, but with the same time step sizes 
% as for uniform meshes.

% % 
% % {\footnotesize
% \bibliographystyle{spmpsci}
% \bibliography{relaxingCFL}
% % }

\end{document}